\documentclass[11pt]{article}
\usepackage[margin=1.1in]{geometry}
\usepackage{amsmath,amssymb,amsthm}
\usepackage{booktabs}
\usepackage{hyperref}
\usepackage{microtype}
\usepackage{tikz}

\newtheorem{theorem}{Theorem}[section]
\newtheorem{lemma}[theorem]{Lemma}
\newtheorem{proposition}[theorem]{Proposition}
\newtheorem{corollary}[theorem]{Corollary}
\theoremstyle{remark}
\newtheorem{remark}[theorem]{Remark}
\theoremstyle{definition}
\newtheorem{definition}[theorem]{Definition}

\newcommand{\Cyc}{\mathcal{C}}
\newcommand{\TC}{\mathrm{TC}}

\title{Small graphs without power-of-two cycles:\\
a lower bound of 24, a correction to a construction of Exoo,\\
and explicit bounds for $f(k)$}
\author{Daniel Garcia\thanks{Independent researcher. Computations were carried out with the assistance of an AI system (Claude); all results were independently re-verified as described in Section~\ref{sec:data}.}}
\date{September 2026}

\begin{document}
\maketitle

\begin{abstract}
The Erd\H{o}s--Gy\'arf\'as conjecture states that every graph with minimum degree at least~3 contains a cycle whose length is a power of two. We prove by a SAT-based exhaustive search that every graph with minimum degree at least~3 on at most 23 vertices contains a cycle of length 4 or a cycle of length 8; consequently any counterexample has at least 24 vertices, improving the previously published bound of 16, and the smallest graph of minimum degree~3 with no 4-cycle and no 8-cycle has exactly 24 vertices. We show that the lemma underlying Exoo's 450-vertex cubic graph with no cycles of length $4,8,16,32$ is false---the Tutte--Coxeter graph contains 8-cycles alternating between outer and chord edges---so that the graph as specified contains 32-cycles; we repair the construction and verify the corrected graph, so that $f(5)\le 450$ stands. We introduce an exact ``window calculus'' for vertex-replacement constructions and use it to prove $f(k)\le 15\,n_3(2^{k-2}+1)$ for all $k\ge 4$, where $n_3(g)$ is the order of the smallest known cubic graph of girth $g$; in particular $f(6)\le 32\,640$, the first bound for $f(6)$. The same calculus shows that Exoo's 78-vertex witness for $f(4)\le 78$ is optimal among all gadget designs on bases with at most 12 vertices, and yields a counting obstruction that rules out bases of girth at most 13 for the $H_{15}$ construction. All graphs are provided, and every level of the exhaustive search is certified by a DRAT proof checked with \texttt{drat-trim}.
\end{abstract}

\section{Introduction}

For a graph $G$ let $\Cyc(G)$ denote its set of cycle lengths. Erd\H{o}s and Gy\'arf\'as conjectured in 1995 that $\Cyc(G)\cap\{4,8,16,\dots\}\neq\emptyset$ for every graph $G$ of minimum degree at least~3 \cite{Erdos97}. The conjecture is open; it is listed as Problem~64 on \emph{erdosproblems.com}. Liu and Montgomery \cite{LiuMontgomery} proved that a power-of-two cycle is forced once the average degree exceeds an absolute (uncomputed) constant, which refutes the stronger belief of Erd\H{o}s and Gy\'arf\'as that no minimum degree suffices. Below that constant, including the cubic case, nothing is known beyond special graph classes and finite searches.

Following Exoo \cite{Exoo}, let $f(k)$ denote the order of a smallest cubic graph with no cycle of length $2^m$ for any $m\le k$. Then $f(2)=10$ (the Petersen graph) and $f(3)=24$ (Markstr\"om \cite{Markstrom}); Exoo gave $54\le f(4)\le 78$, the lower bound being an unpublished computation of Markstr\"om, and $f(5)\le 450$. We note the elementary reformulation
\begin{equation}\label{eq:reform}
\text{the conjecture is false} \iff f(k)\le 2^k \text{ for some } k,
\end{equation}
since a graph on at most $2^k$ vertices avoiding all $2^m$-cycles with $m\le k$ avoids every power of two it could contain.

On the side of small counterexamples, the published state of the art is that any counterexample has at least 16 vertices, obtained by Royle's search through 15 vertices reported in \cite{Markstrom}; the figure ``17'' that circulates in secondary sources is not supported by any primary source we could locate. For cubic graphs Markstr\"om showed that all cubic graphs on at most 28 vertices contain a 4-, 8- or 16-cycle. Recent certified searches in this genre include Tranquilli's result that every cubic bipartite graph on at most 58 vertices contains a cycle of length 4, 8 or 16 \cite{Tranquilli}, and Carr's structural results on minimal counterexamples \cite{Carr}.

\paragraph{Results.}
\begin{itemize}
\item \textbf{Theorem~\ref{thm:A}.} Every graph with minimum degree at least 3 on at most 23 vertices contains a cycle of length 4 or a cycle of length 8. Hence any counterexample to the Erd\H{o}s--Gy\'arf\'as conjecture has at least 24 vertices, and the smallest graph of minimum degree~3 with no 4-cycle and no 8-cycle has exactly 24 vertices.
\item \textbf{Proposition~\ref{prop:exoo}.} The Tutte--Coxeter graph contains 8-cycles with no two consecutive outer edges, contradicting the lemma used in \cite{Exoo}; the 450-vertex graph of \cite{Exoo}, as specified, contains a 32-cycle. \textbf{Theorem~\ref{thm:450}.} A corrected orientation yields a cubic graph on 450 vertices with no cycle of length $4,8,16$ or $32$, so $f(5)\le 450$.
\item \textbf{Theorem~\ref{thm:B}.} $f(k)\le 15\,n_3(2^{k-2}+1)$ for all $k\ge 4$; in particular $f(6)\le 32\,640$.
\item \textbf{Proposition~\ref{prop:78}.} Over all $C_4$-free cubic base graphs on at most 12 vertices and the gadget library $\{$vertex, triangle, $H_7$, $H_7'\}$, the minimum order of a $\{4,8,16\}$-free expansion is 78, attained only on Exoo's base.
\item \textbf{Lemma~\ref{lem:count}.} A counting obstruction for $H_{15}$-expansions, which rules out all bases of girth at most 13.
\end{itemize}

\section{The lower bound}

\begin{theorem}\label{thm:A}
Every graph with minimum degree at least $3$ on at most $23$ vertices contains a cycle of length $4$ or a cycle of length $8$.
\end{theorem}

\begin{corollary}
Every counterexample to the Erd\H{o}s--Gy\'arf\'as conjecture has at least $24$ vertices. The smallest graph with minimum degree at least $3$ containing neither a $4$-cycle nor an $8$-cycle has exactly $24$ vertices.
\end{corollary}

\begin{proof}[Proof of the corollary]
A counterexample has no 4-cycle and no 8-cycle, so by the theorem it has at least 24 vertices. Markstr\"om's four cubic graphs on 24 vertices \cite{Markstrom} have no 4- or 8-cycle, giving the upper bound.
\end{proof}

\subsection{Method}

For each $n$ we decide by SAT whether a graph on the vertex set $\{0,\dots,n-1\}$ with minimum degree at least 3 and no cycle of length 4 or 8 exists. The encoding uses one Boolean variable $e_{uv}$ per pair $u<v$ and the following clauses:
\begin{enumerate}
\item for every vertex $v$, a cardinality constraint $\sum_{u\ne v} e_{uv}\ge 3$ (sequential-counter encoding);
\item for every 4-subset $\{a,b,c,d\}$ and each of its three cyclic pairings, the clause $\neg e_{pq}\vee\neg e_{qr}\vee\neg e_{rs}\vee\neg e_{sp}$ forbidding that 4-cycle;
\item symmetry breaking: for each $i<n-1$, the row of $i$ is lexicographically at least the row of $i+1$ on the columns other than $i,i+1$, encoded with auxiliary variables $d_{i,t}\to(e_{i,k_t}\wedge\neg e_{i+1,k_t})$ and clauses $d_{i,0}\vee\dots\vee d_{i,t-1}\vee e_{i,k_t}\vee\neg e_{i+1,k_t}$;
\item 8-cycles are excluded lazily: whenever the solver returns a model, up to 400 of its 8-cycles are enumerated and each is blocked by the clause $\bigvee_{j}\neg e_{x_j x_{j+1}}$ over its eight edges; the solver is then resumed incrementally.
\end{enumerate}
The loop terminates either with a model having no 8-cycle (a witness, which is re-verified independently) or with UNSAT.

\emph{Soundness of UNSAT.} Every graph in the target class satisfies every clause: (1) and (2) by definition; a blocking clause in (4) states that eight specific edges forming an 8-cycle are not all present, which any 8-cycle-free graph satisfies; and the lex-maximal labelling of any graph satisfies (3), so each isomorphism class retains a satisfying labelling. Hence UNSAT proves the class empty. The solver used is CaDiCaL~1.9.5 through PySAT. The pipeline was validated by reproducing the known answers for $n\le 15$ \cite{Markstrom}, and every witness produced in other runs (Section~\ref{sec:closed}) was verified by an independent cycle-enumeration routine that was itself cross-checked against brute-force enumeration on graphs with known cycle spectra.

\begin{table}[h]
\centering\small
\begin{tabular}{rrrr}
\toprule
$n$ & CEGAR rounds & 8-cycles blocked & time (s) \\
\midrule
$\le 9$ & 1 & 0 & $<0.1$ \\
10 & 8 & 94 & $<0.1$ \\
11 & 19 & 305 & $<0.1$ \\
12 & 44 & 920 & 0.1 \\
13 & 96 & 2\,594 & 0.2 \\
14 & 179 & 5\,825 & 1.6 \\
15 & 338 & 12\,155 & 3.8 \\
16 & 612 & 23\,143 & 4.1 \\
17 & 1\,057 & 43\,760 & 10.3 \\
18 & 1\,765 & 74\,479 & 30.4 \\
19 & 2\,852 & 121\,319 & 83.7 \\
20 & 4\,662 & 198\,519 & 171 \\
21 & 7\,567 & 317\,103 & 584 \\
22 & 12\,508 & 496\,446 & 2\,085 \\
23 & 15\,029 & 770\,552 & $\approx 5\,850$ \\
\bottomrule
\end{tabular}
\caption{The search of Theorem~\ref{thm:A}: all levels UNSAT. Single core of a consumer desktop. For $n\le 9$ the instance is unsatisfiable before any 8-cycle is blocked: no $C_4$-free graph of minimum degree 3 exists on fewer than 10 vertices.}
\end{table}

\begin{remark}
The theorem is stronger than what the conjecture requires (no condition on 16-cycles is used). The search assumes no minimal-counterexample reductions; the reductions of \cite{Markstrom,Carr} (vertices of degree at least 4 form an independent set; every vertex is adjacent to a vertex of degree 3) are sound for extending the ladder and are the natural lever for pushing past 24, where the cost, which grows by a factor of roughly 3--4 per vertex, becomes prohibitive for the full $\{4,8,16\}$ decision.
\end{remark}

\section{The window calculus}\label{sec:window}

All known small graphs without short power-of-two cycles are obtained by replacing the vertices of a base graph by gadgets. We make the bookkeeping exact.

\begin{definition}
A \emph{vertex gadget} is a graph $H$ with three distinguished \emph{attachment} vertices of degree 2, all other vertices of degree 3. Replacing a vertex $x$ of a cubic graph $B$ by $H$ means deleting $x$ and joining its three former neighbours to the three attachments by a bijection. For attachments $p,q$ let $S_H(p,q)$ be the set of lengths of simple $p$--$q$ paths in $H$.
\end{definition}

\begin{lemma}[Window lemma]\label{lem:window}
Let $B$ be cubic and let $G$ be obtained by replacing every vertex $x$ of $B$ by a gadget $H_x$. Then a cycle of $G$ either lies inside a single gadget, or projects onto a cycle $x_1\cdots x_\ell$ of $B$ and has length $\ell+\sum_{i=1}^{\ell}s_i$ with $s_i\in S_{H_{x_i}}(p_i,q_i)$, where $p_i,q_i$ are the attachments facing $x_{i-1}$ and $x_{i+1}$. Conversely, every such sum is the length of a cycle of $G$. In particular $G$ has a cycle of length $T$ if and only if some gadget has an internal cycle of length $T$ or $T-\ell$ lies in the Minkowski sum $\sum_i S_{H_{x_i}}(p_i,q_i)$ for some cycle of $B$ of length $\ell$.
\end{lemma}

\begin{proof}
A gadget copy is joined to the rest of $G$ by exactly three edges, so a cycle not contained in it uses either zero or two of them, i.e.\ it visits the copy at most once, along a simple path between two attachments. Contracting each copy maps the cycle onto a closed walk of $B$ in which no vertex repeats, hence a cycle; the length is as stated. Conversely, given a cycle of $B$ and a choice of internal path for each visited copy, the union is a cycle of $G$, and the choices are independent because the copies are disjoint.
\end{proof}

The gadgets used by Exoo are the following. $H_7$ has vertices $u,v,w$ (attachments) and $a,b,c,d$ and edges $va,ab,bw,cv,wd,ac,cu,ud,db$. $H_{15}$ consists of two copies $A,B$ of $H_7$ and a vertex $z$ with edges $v_Av_B$, $w_Az$, $zw_B$; its attachments are $u=z$, $v=u_B$, $w=u_A$. By exhaustive enumeration,
\[
S_{H_7}(u,v)=S_{H_7}(u,w)=\{2,\dots,6\},\quad S_{H_7}(v,w)=\{3,\dots,6\},\quad \Cyc(H_7)=\{3,5,6,7\},
\]
\[
S_{H_{15}}(u,v)=S_{H_{15}}(u,w)=\{3,\dots,14\},\qquad S_{H_{15}}(v,w)=\{5,\dots,14\},
\]
\[
\Cyc(H_{15})=\{3,5,6,7,9,\dots,15\}.
\]
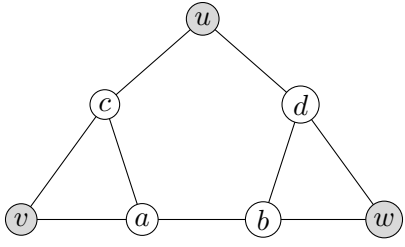
\begin{figure}[ht]
\centering
\begin{tikzpicture}[scale=1.6,every node/.style={circle,draw,inner sep=1.6pt,fill=white}]
\coordinate (v) at (0,0); \coordinate (a) at (1,0); \coordinate (b) at (2,0); \coordinate (w) at (3,0);
\coordinate (c) at (0.691,0.951); \coordinate (u) at (1.5,1.66); \coordinate (d) at (2.309,0.951);
\draw (v)--(a)--(b)--(w) (c)--(v) (w)--(d) (a)--(c)--(u)--(d)--(b);
\node[fill=black!15] at (v) {$v$}; \node[fill=black!15] at (w) {$w$}; \node[fill=black!15] at (u) {$u$};
\node at (a) {$a$}; \node at (b) {$b$}; \node at (c) {$c$}; \node at (d) {$d$};
\end{tikzpicture}
\caption{The gadget $H_7$; the shaded vertices are the attachments. Simple $u$--$v$ and $u$--$w$ paths have lengths $2,\dots,6$; $v$--$w$ paths have lengths $3,\dots,6$; the internal cycles have lengths $3,5,6,7$.}
\label{fig:h7}
\end{figure}

Both spectra avoid powers of two, and all path sets are intervals, so the Minkowski sums are intervals: a base cycle of length $\ell$ in which $a$ of the visits are through the attachment $u$ (``$u$-type'') contributes exactly the lengths $[\ell+\Sigma_{\min},\,\ell+\Sigma_{\max}]$ with, for $H_{15}$, $\Sigma_{\min}=3a+5(\ell-a)=5\ell-2a$ and $\Sigma_{\max}=14\ell$.

\section{Exoo's construction for \texorpdfstring{$f(5)$}{f(5)}}

Let $\TC$ be the Tutte--Coxeter graph drawn as in \cite{Exoo}: an outer Hamiltonian cycle $0,1,\dots,29$ and the fifteen chords $\{29+6t,12+6t\},\{28+6t,7+6t\},\{26+6t,3+6t\}$, $t=0,\dots,4$ (indices mod 30). Exoo replaces every vertex by $H_{15}$ with $u$ facing the chord, obtaining a 450-vertex graph that we denote $G_{450}$ (it is not named in \cite{Exoo}), and argues that no 32-cycle arises because ``any 8-cycle in Tutte--Coxeter contains at least two consecutive edges on the outer Hamiltonian cycle''.

\begin{proposition}\label{prop:exoo}
The vertices $0,17,18,5,6,23,22,1$ form an $8$-cycle of $\TC$ whose edges alternate between chords and outer edges. Consequently the graph $G_{450}$ of \cite{Exoo}, with $u$ on every chord, contains a $32$-cycle.
\end{proposition}

\begin{proof}
The edges $\{0,17\},\{18,5\},\{6,23\},\{22,1\}$ are the chords with $t=3,1,4,4$ respectively, and $\{17,18\},\{5,6\},\{23,22\},\{1,0\}$ are outer edges; the eight vertices are distinct. The failure does not depend on the drawing: $\TC$ has exactly 144 Hamiltonian cycles, and for none of them is it true that every 8-cycle contains two consecutive edges of the Hamiltonian cycle (checked by exhaustive enumeration). In $G_{450}$ every visit of this cycle enters a copy of $H_{15}$ through $u$ (the chord side) and leaves through $v$ or $w$, and $3\in S_{H_{15}}(u,v)=S_{H_{15}}(u,w)$; by Lemma~\ref{lem:window} the length $8+8\cdot 3=32$ is attained. (A SAT search on the reconstructed graph produces such a cycle explicitly.)
\end{proof}

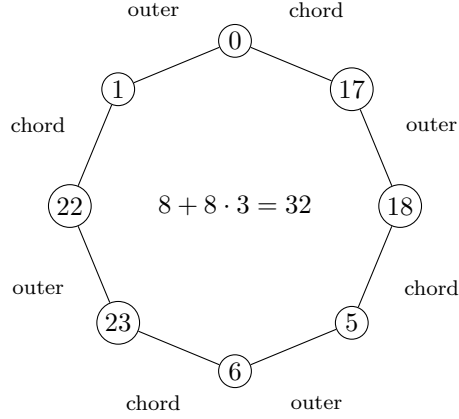
\begin{figure}[ht]
\centering
\begin{tikzpicture}[scale=1.15]
\foreach \i/\lab in {0/0,1/17,2/18,3/5,4/6,5/23,6/22,7/1}{
  \node[circle,draw,inner sep=1.5pt,fill=white] (n\i) at ({90-45*\i}:1.9) {\small$\lab$};
}
\foreach \i/\j/\t in {0/1/chord,1/2/outer,2/3/chord,3/4/outer,4/5/chord,5/6/outer,6/7/chord,7/0/outer}{
  \draw (n\i)--(n\j);
}
\foreach \i/\t/\ang in {0/chord/67.5,1/outer/22.5,2/chord/-22.5,3/outer/-67.5,4/chord/-112.5,5/outer/-157.5,6/chord/157.5,7/outer/112.5}{
  \node[font=\scriptsize] at ({\ang}:2.45) {\t};
}
\node[font=\small] at (0,0) {$8+8\cdot 3=32$};
\end{tikzpicture}
\caption{An 8-cycle of the Tutte--Coxeter graph, in the labelling of \cite{Exoo}, whose edges alternate between chords and outer edges. With $u$ on every chord, each of its eight gadget crossings can have length 3, producing a 32-cycle in $G_{450}$.}
\label{fig:alt}
\end{figure}

\begin{theorem}\label{thm:450}
There is an orientation of the $H_{15}$-replacement of $\TC$ that yields a cubic graph $G'_{450}$ on $450$ vertices with no cycle of length $4$, $8$, $16$ or $32$. Hence $f(5)\le 450$.
\end{theorem}

\begin{proof}
Choose for each vertex of $\TC$ which incident edge faces $u$ so that no 8-cycle of $\TC$ has all eight of its vertices' $u$-edges on the cycle. This is a satisfiable SAT instance with 90 clauses (one per 8-cycle); a solution is listed in the appendix. For the resulting graph, Lemma~\ref{lem:window} gives: a base 8-cycle has at least one $v$--$w$ visit, so its expansions have length at least $8+7\cdot 3+5=34$; every other base cycle has $\ell\ge 10$ ($\TC$ is bipartite of girth 8) and expands to length at least $10\cdot(1+3)=40$; and $\Cyc(H_{15})$ contains no power of two. Thus no cycle of length $4,8,16$ or $32$ exists. This was confirmed computationally: exhaustive search finds no cycle of length 4, 8 or 16, and a SAT encoding of ``a cycle of length exactly 32'' is unsatisfiable.
\end{proof}

\begin{remark}
The same reconstruction confirmed Exoo's $G_{78}$ (base: the Petersen graph with one vertex replaced by a triangle, all vertices but one triangle vertex replaced by $H_7$) and $G_{420}$ (the truncated icosahedron with every vertex replaced by $H_7$, $u$ on the edges bordering two hexagons): both contain no cycle of length 4, 8 or 16. For $G_{78}$ the orientation is not determined by the figure in \cite{Exoo}; a literal reading of the markers yields a graph with a 16-cycle, and an exhaustive search shows that exactly 16 of the $3^{11}$ orientations are valid.
\end{remark}

\section{Upper bounds for \texorpdfstring{$f(k)$}{f(k)}}

Let $n_3(g)$ denote the order of the smallest known cubic graph of girth $g$ (values are tabulated in \cite{ExooJajcay}; $n_3(17)=2176$).

\begin{theorem}\label{thm:B}
For every $k\ge 4$, $f(k)\le 15\,n_3(2^{k-2}+1)$. In particular $f(6)\le 15\cdot 2176=32\,640$.
\end{theorem}

\begin{proof}
Let $B$ be cubic of girth $g>2^{k-2}$ and replace every vertex by $H_{15}$ (any orientation). By Lemma~\ref{lem:window} a cycle of the expansion either lies in a copy of $H_{15}$, whose spectrum contains no power of two, or projects onto a base cycle of length $\ell\ge g$ and has length at least $\ell+3\ell=4\ell\ge 4g>2^k$. Hence no cycle of length $2^m$ with $m\le k$ exists, and the expansion has $15|B|$ vertices.
\end{proof}

Previously the only route to a finite bound on $f(k)$ was a cubic graph of girth exceeding $2^k$, of order about $2^{(3/4)2^k}$ with the best known explicit families; Theorem~\ref{thm:B} divides the exponent by four. By \eqref{eq:reform} the conjecture is equivalent to $f(k)>2^k$ for all $k$; the bounds here are far from that threshold.

Below girth 17 the short base cycles must be handled by the orientation. For the target $64$, a base $\ell$-cycle with $a$ $u$-type visits expands to lengths at least $6\ell-2a$, so it avoids 64 iff $a\le 3\ell-33$: at most $6,9,12,15$ $u$-type visits on cycles of length $13,14,15,16$, and no condition for $\ell\ge 17$.

\begin{lemma}[Counting obstruction]\label{lem:count}
Let $B$ be cubic and, for a vertex $v$ and $\ell\in\{13,\dots,16\}$, let $c^{\ell}_{\min}(v)$ be the minimum over the three edges at $v$ of the number of $\ell$-cycles of $B$ containing that edge. If $\sum_v c^{\ell}_{\min}(v)>(3\ell-33)\,N_\ell(B)$ for some $\ell$, where $N_\ell$ is the number of $\ell$-cycles, then no orientation of the $H_{15}$-replacement of $B$ avoids $64$-cycles.
\end{lemma}

\begin{proof}
$\sum_C a(C)=\sum_v\#\{\ell\text{-cycles through the $u$-edge of }v\}\ge\sum_v c^{\ell}_{\min}(v)$, while avoiding 64 requires $a(C)\le 3\ell-33$ for every $\ell$-cycle $C$.
\end{proof}

The Tutte 12-cage is edge-transitive with 1008 twelve-cycles, so each edge lies on 64 of them and $\sum_v c_{\min}=126\cdot 64=8064>3\cdot 1008$; the Balaban 11-cage fails likewise ($3\ell-33=0$ for $\ell=11$). For girth 13, Hoare's 272-vertex Cayley graph of $\mathrm{AGL}(1,17)$ \cite{ExooJajcay}---which we reconstructed by searching all generating pairs $\{t,g,g^{-1}\}$, $t$ an involution; 544 pairs give girth 13---has 544 thirteen-cycles and every edge lies on at least 16 of them, so $\sum_v c_{\min}\ge 4352>6\cdot 544=3264$. The corresponding SAT instances are unsatisfiable, as they must be. The same search over $\mathrm{AGL}(1,p)$ for $p=23,29,31,37$ produced girth-14 Cayley graphs on $506,812,930,1332$ vertices; the first fails Lemma~\ref{lem:count} narrowly ($16\,192>15\,939$), the others pass it, and their orientation instances (which would give $f(6)\le 12\,180$ and $13\,950$) were undecided after several CPU-hours. We leave them open.

\section{Optimality of 78 and gadget censuses}

Let the gadget library consist of: leaving a vertex unreplaced (path sets $\{0\}$), replacing it by a triangle (each pair of attachments joined by paths of lengths $\{1,2\}$), $H_7$ in its three orientations, and the symmetric 7-vertex gadget $H_7'$ with all three path sets $\{2,\dots,6\}$.

\begin{proposition}\label{prop:78}
Among all designs on $C_4$-free cubic base graphs with at most $12$ vertices using this library, the minimum order of an expansion with no cycle of length $4$, $8$ or $16$ is $78$; it is attained only on Exoo's $12$-vertex base.
\end{proposition}

\begin{proof}
A base 4-cycle is fatal for every assignment (its window always contains 8 or 16), so bases may be taken $C_4$-free; there are 3 such cubic graphs on 10 vertices and 8 on 12 (enumerated by SAT, matching the count of three $C_4$-free cubic graphs on 10 vertices in \cite{Exoo}). For each base, a backtracking search over assignments with exact Minkowski windows (Lemma~\ref{lem:window}) against $\{4,8,16\}$ finds the minimum; the minimal design was rebuilt explicitly and verified by exhaustive cycle search.
\end{proof}

\begin{proposition}
\begin{enumerate}
\item There is no vertex gadget on $5$ or $9$ vertices without $4$- or $8$-cycles; on $7$ vertices there are exactly two, $H_7$ and $H_7'$; on $11$ vertices there is none whose attachments are pairwise at distance at least $3$.
\item Among two-attachment gadgets (spliced into edges) on $4,6,8$ vertices there are $1,4,19$ up to isomorphism, none without a $4$- or $8$-cycle; on $10$ vertices there is none with attachments at distance at least $3$.
\end{enumerate}
\end{proposition}

The second statement bears on a natural route to a counterexample: a two-attachment gadget whose path-length set $S$ satisfies $\max S<2\min S$ (for instance $S=\{5,7\}$), spliced into every edge of a base whose even cycle lengths avoid certain intervals, would give a graph with no power-of-two cycle at all. Every such gadget found so far contains a 4- or 8-cycle. We conjecture that a two-attachment gadget without internal power-of-two cycles always has $\max S\ge 2\min S$.

\section{Further closed routes}\label{sec:closed}

\begin{proposition}
Every graph with minimum degree at least $3$ on at most $19$ vertices contains a cycle of length $4$, $6$, $10$ or $12$.
\end{proposition}
This is the same search as Theorem~\ref{thm:A} with 6-, 10- and 12-cycles blocked lazily; each level is unsatisfiable within seconds. It shows that a base for the $\{5,7\}$-gadget route above would need at least 20 vertices.

\begin{remark}[Congruence obstructions]
No graph of minimum degree at least 3 has all cycle lengths divisible by an integer $m\ge 3$. Indeed, take a deepest vertex $v$ of a depth-first search tree; its two non-tree neighbours $x,y$ are ancestors, and with $A=d(x,v)$, $B=d(x,y)$ along the tree one finds cycles of lengths $A+1$, $A+B+1$, $B+2$. If $m$ divides the first two it divides $B$, and then it does not divide $B+2$. This folklore argument rules out the ``all lengths $\equiv 0 \pmod m$'' family of would-be counterexamples for every $m$.
\end{remark}

\section{Data and verification}\label{sec:data}

All graphs are provided in graph6 format: the corrected $G'_{450}$, the reconstructed $G_{78}$ and $G_{420}$, Hoare's girth-13 graph, and the girth-14 Cayley graphs. Cycle-length claims were checked by two independent methods: a rooted depth-first search with distance pruning (exhaustive; used for lengths up to 16), and a SAT encoding of ``there is a cycle of length exactly $L$'' (used for lengths 32 and above; unsatisfiability certifies absence). The depth-first routine was validated against brute-force enumeration on graphs with known spectra (Petersen, Heawood, dodecahedron, random cubic graphs). Every level $n=4,\dots,23$ of Theorem~\ref{thm:A} was additionally certified independently of the incremental search and of the Python wrapper: the static formula consisting of the base encoding together with all blocked 8-cycles of that level was solved once by a standalone CaDiCaL~2.1.3 with DRAT proof logging, and the proof was checked with \texttt{drat-trim}. All twenty proofs were accepted (\texttt{s VERIFIED}); the largest, for $n=23$, has $800\,757$ clauses and a 3.1~GB proof, checked in 74 minutes. All SAT instances of Section~\ref{sec:closed} can be re-run from the accompanying scripts; the ladder produces a checkpoint file of blocked cycles from which any level's certificate is regenerated.

\appendix
\section{The repaired orientation of \texorpdfstring{$G'_{450}$}{G'450}}

With $\TC$ labelled as in Section~4, the following table gives, for each vertex $x$, the neighbour $y$ such that the edge $xy$ is attached to $u$ of the copy of $H_{15}$ replacing $x$; the other two edges are attached to $v$ and $w$ in either order (the two choices give isomorphic graphs, since $H_{15}$ has an automorphism exchanging $v$ and $w$).

\begin{center}\small
\begin{tabular}{rr@{\qquad}rr@{\qquad}rr@{\qquad}rr@{\qquad}rr@{\qquad}rr}
\toprule
$x$&$y$&$x$&$y$&$x$&$y$&$x$&$y$&$x$&$y$&$x$&$y$\\
\midrule
0&1 & 5&4 & 10&9 & 15&8 & 20&19 & 25&16\\
1&0 & 6&5 & 11&10 & 16&15 & 21&14 & 26&3\\
2&1 & 7&6 & 12&11 & 17&0 & 22&1 & 27&20\\
3&2 & 8&7 & 13&4 & 18&19 & 23&6 & 28&7\\
4&3 & 9&2 & 14&13 & 19&20 & 24&11 & 29&28\\
\bottomrule
\end{tabular}
\end{center}
Twelve of the thirty $u$-edges are chords and eighteen are outer edges; every one of the 90 eight-cycles of $\TC$ contains a vertex whose $u$-edge is off the cycle.

\section*{Acknowledgments}
The author thanks Geoffrey Exoo for his encouraging correspondence about the correction in Section~4. The data and certificates accompanying this paper are archived at \url{https://doi.org/10.5281/zenodo.22180583}.

\end{document}